\documentclass[10pt,unicode,leqno]{article}

\usepackage{amsmath,amssymb,amsthm,amsfonts}

\newcommand{\R}{\mathbb{R}}
\newcommand{\IR}{\mathbb{IR}}

\newcommand{\KR}{\mathbb{KR}}
\newcommand{\inum}[1]{\mathbf{#1}}
\newcommand{\dual}[1]{\text{\rm dual}(#1)}

\newtheorem{theorem}{Theorem}
\newtheorem{proposition}{Proposition}
\newtheorem{corollary}{Corollary}
\newtheorem{example}{Example}

\begin{document}

\begin{center}
\uppercase{\bf On a Class of Parameterized Solutions to Interval
Parametric Linear Systems}
\bigskip

Evgenija D.\ Popova
\end{center}
\bigskip

\begin{center}
\begin{minipage}[c]{12cm}
\small {\bf Abstract}
\smallskip

Presented is a new method yielding parameterized solution to an
interval parametric linear system. Some properties of this method
are discussed. The solution enclosure it provides is compared to the
enclosures by other methods.
It is shown that an application, proposed by other authors, cannot be done in the general case. \\[-6pt]

\smallskip

\end{minipage}
\end{center}
\bigskip

{\bf 1. Introduction.} Parameterized solutions to interval
parametric linear systems are linear functions of interval
parameters that estimate the united solution set. Parameterized
solutions present an alternative form of the traditional numerical
interval vectors enclosing the solution set. They allow subsequent
problems involving both the primary solution and the initial
interval parameters to account better for the parameter
dependencies, cf. \cite{Kol2014} \cite{PopovaED:2018ArXiv}. Methods
for deriving parameterized solutions are developed in relation to
many classical interval methods yielding interval vectors enclosing
the united solution set, see
\cite{Kol2014}--\cite{PopovaED:2017:POE} and the references therein
mentioning most of the works on parameterized solutions. Recently,
Kolev presented in \cite{Kolev2018} parameterized analogue of the
generalized method of Neumaier and Pownuk \cite{NePow},
\cite{Pop2014CAMWA} and proposed an application of it. A
methodological framework alternative to \cite{NePow},
\cite{Pop2014CAMWA} was proposed in \cite{PopovaED:2018} and its
advantages for handling a class of interval parametric linear
systems were demonstrated. A new parameterized solution, based on
the methodology of  \cite{PopovaED:2018} and different from the
parameterized form (\ref{x(p,l)}) was proposed in
\cite{PopovaED:2018ArXiv} along with a new application direction.

Present work is motivated by  \cite{Kolev2018}. While the
parameterized solution in  \cite{Kolev2018} is based on affine
arithmetic, in Section 3 of this work we propose a parameterized
solution in form (\ref{x(p,l)}) which is based on the numerical
method in  \cite{PopovaED:2018} and not using affine arithmetic. We
show that the two forms of parameterized solutions related to the
numerical method in \cite{PopovaED:2018} can be applied also to the
numerical methods in \cite{NePow}, \cite{Pop2014CAMWA} without using
affine arithmetic. The methodology in  \cite{PopovaED:2018} has an
expanded scope of applications and provides a sharper solution
enclosure than most of the methods for a wide class of parametric
systems involving rank one uncertainty structure. In this work we
demonstrate that for the latter class of parametric systems, the
proposed here parameterized solution provides sharper solution
enclosure than a variety of parameterized solutions based on affine
arithmetic and compared in \cite{SkalnaHla19}. In Section 4 we
discuss in details and demonstrate by a numerical example that the
application proposed in \cite{Kolev2018} cannot be done to arbitrary
parametric linear systems with rank one uncertainty structure.

{\bf 2. Preliminaries.} Denote by $\R^{m\times n}$  the set of real
$m\times n$ matrices. Vectors are considered as one-column matrices.
A real compact interval is $\inum{a} = [a^-, a^+] := \{a\in\R \mid
a^-\leq a\leq a^+\}$ and $\IR^{m\times n}$ denotes the set of
interval $m\times n$ matrices. For $\inum{a} = [a^-, a^+]$, define
its mid-point $\check{a}:= (a^- + a^+)/2$,  the radius $\hat{a} :=
(a^+ - a^-)/2$ and the magnitude $|\inum{a}|:= \max\{|a^-|,
|a^+|\}$. These functions are applied to interval vectors and
matrices componentwise. The inequalities are understood
componentwise. The spectral radius of a matrix $A\in\R^{n\times n}$
is denoted by $\varrho(A)$. The identity matrix of appropriate
dimension is denoted by $I$. For $A_k\in\R^{n\times m}$, $1\leq
k\leq t$, $\left(A_1,\ldots,A_t\right)\in\R^{n\times tm}$ denotes
the matrix obtained by stacking the columns of the matrices $A_k$.
Denote the $i$-th column of $A\in\R^{n\times m}$ by $A_{\bullet i}$
and its $i$-th row by $A_{i\bullet}$.

We consider systems of linear algebraic equations with linear
uncertainty structure
\begin{equation} \label{pls}
\begin{split}
&A(p) x \; = \; a(p), \quad p\in\inum{p}\in\IR^K,\\ &A(p) := A_0 +
\sum_{k = 1}^K p_kA_k, \qquad  a(p)  := a_0 + \sum_{k = 1}^K p_ka_k,
\end{split}
\end{equation}
where $A_k\in\R^{n\times n}$, $a_k\in\R^n$, $k = 0,\ldots, K$ and
the parameters $p=(p_1,\ldots$, $p_K)^\top$ are considered to be
uncertain and varying within given non-degenerate\footnote{An
interval $\inum{a}=[a^-, a^+]$ is degenerate if $a^-=a^+$.}
intervals $\inum{p}=(\inum{p}_1,\ldots,\inum{p}_K)^\top$. Nonlinear
dependencies between interval valued parameters in linear algebraic
systems are usually linearized to the form (\ref{pls}) and methods
for the latter are applied to bound the corresponding solution set.
The so-called united parametric  solution set of the system
(\ref{pls}) is defined by $$\Sigma^p_{\rm uni}=\Sigma_{\rm
uni}(A(p),a(p),\inum{p}) := \{x\in\R^n \mid (\exists
p\in\inum{p})(A(p)x=a(p))\}.$$ Usually, the interval methods
(providing interval enclosure of $\Sigma^p_{\rm uni}$) generate
numerical interval vectors that contain $\Sigma^p_{\rm uni}$. A new
type -- parameterized solution -- is proposed in \cite{Kol2014}.
This solution is in form of an affine-linear function of
interval-valued parameters
\begin{equation}\label{x(p,l)}
x(p,r) = \tilde{x} + V(\check{p}-p) + r, \quad p\in\inum{p}, \;
r\in\inum{r}= [-\hat{r},\hat{r}],
\end{equation}
where $\tilde{x}, \hat{r}\in\R^n$, $V\in\R^{n\times K}$. Some
representations move $\tilde{x}$ into the interval vector $\inum{r}$
and consider the parameters $p,r$ varying independently within the
interval $[-1,1]$. The parameterized solution has the property
$\Sigma^p_{uni}\subseteq x(\inum{p},\inum{r})$, where
$x(\inum{p},\inum{r})$ is the interval hull of $x(p,r)$ over
$p\in\inum{p}$, $r\in\inum{r}$. For a nonempty and bounded set
$\Sigma\subset\R^n$, its interval hull is $\square\Sigma := \bigcap
\{\inum{x}\in\IR^n \mid \Sigma\subseteq \inum{x}\}$.

In what follows we consider another form of the parametric system
(\ref{pls}) and some numerical and parameterized solutions related
to this form. Let ${\cal K}=\{1,\ldots,K\}$ and $\pi',\pi''$ be two
subsets of ${\cal K}$ such that $\pi'\cap\pi''=\emptyset$,
$\pi'\cup\pi''={\cal K}$, Card$(\pi')=K_1$. The permutation $\pi'$
denotes the indices of the parameters that appear in both the matrix
and the right-hand side of the system, while $\pi''$ involves the
indices of the parameters that appear only in $a(p)$ in (\ref{pls}).
Denote $p_\pi = (p_{\pi_1},\ldots,p_{\pi_K})$ and by $D_{p_{\pi}}$ a
diagonal matrix with diagonal vector $p_{\pi}$. %
The system (\ref{pls}) has the following equivalent form
\begin{equation}\label{plsG}
\left(A_0+LD_{g(p_{\pi'})}R\right)x = a_0+LD_{g(p_{\pi'})}t +
Fp_{\pi''}, \quad p\in\inum{p},
\end{equation}
where $g(p_{\pi'})\in\R^\gamma$, $\gamma=\sum_{k=1}^{K_1} \gamma_k$,
$g(p_{\pi'}) = \left(g_1^\top (p_{\pi'_1}),\ldots, g_{K_1}^\top
(p_{\pi'_{K_1}})\right)^\top$, $L=\left(L_1,\ldots,
L_{K_1}\right)\in\R^{n\times\gamma}$,
$R=\left(R_1^\top,\ldots,R_{K_1}^\top\right)^\top\in\R^{\gamma\times
n}$ and for $1\leq k\leq K_1$,
$g_k(p_{\pi'_k})=(p_{\pi'_k},\ldots,p_{\pi'_k})^\top\in\R^{\gamma_k}$,
$p_{\pi'_k}A_{\pi'_k} = L_kD_{g_k(p_{\pi'_k})}R_k$, and
$\sum_{k\in\pi'} p_ka_k = LD_{g(p_{\pi'})}t$. We assume that
(\ref{plsG}) provides an equivalent optimal rank one representation
(cf. \cite{PopovaED:2018}) of either $A(p_{\pi'})-A_0$, or of
$A^\top(p_{\pi'})-A_0^\top$.  Every interval parametric linear
system (\ref{pls}) has an equivalent, optimal, rank one
representation (\ref{plsG}) and there are various ways to obtain it,
cf.~\cite{Piziak}, \cite{PopovaED:2018}. The following theorem
presents a method for computing numerical interval enclosure of
$\Sigma^p_{uni}$.

\begin{theorem}[{\cite{PopovaED:2018}}] \label{solution1}
Let (\ref{plsG}) be the equivalent optimal rank one representation
of system (\ref{pls}) and let the matrix $A(\check{p})$ be
nonsingular. Denote $C=A^{-1}(\check{p})$ and
$\check{x}=Ca(\check{p})$. If
\begin{equation}\label{strReg2}
\varrho\left(\left|(RCL)D_{g(\check{p}_{\pi'}-\inum{p}_{\pi'})}\right|\right)<1,
\end{equation}
\begin{itemize}
\item[(i)] $\Sigma_{uni}\left(A(p),a(p), \inum{p}\right)$ and the
solution set $\Sigma_{\rm uni}((\ref{eqY}))$ of system (\ref{eqY})
are bounded
\begin{equation}\label{eqY}
\left(I-RCLD_{g(p_{\pi'})}\right)y = R\check{x} - RCFp_{\pi''} -
RCLD_{g(p_{\pi'})}t, \quad  p\in [-\hat{p},\hat{p}],
\end{equation}
\item[(ii)] $\inum{y}\supseteq\Sigma_{\rm uni}((\ref{eqY}))$
is computable by methods that require (\ref{rhoDelta}) {\rm
(cf.\cite{PopovaED:2018})}
\begin{equation}\label{rhoDelta}
\varrho
(\sum_{i=1}^K\left|\left(A(\check{p})\right)^{-1}A_i\right|\hat{p}_i)<1,
\end{equation}
\item[(iii)]
every $x\in\Sigma_{uni}\left(A(p), a(p), \inum{p}\right)$ satisfies
\begin{equation}\label{xx}
x  \in \check{x} - (CF)[-\hat{p}_{\pi''}, \hat{p}_{\pi''}] +
(CL)\left(D_{g([-\hat{p}_{\pi'},
\hat{p}_{\pi'}])}|\inum{y}-t|\right).
\end{equation}
\end{itemize}
\end{theorem}
The condition (\ref{rhoDelta}) is weaker and holds true when the
condition (\ref{strReg2}) is satisfied, cf. \cite{PopovaED:2018}.
The interval vector $\inum{y}$ in Theorem \ref{solution1} (ii) can
be obtained by a variety of numerical methods, many of them are
discussed in \cite{PopovaED:2018}.

\begin{theorem}[\cite{PopovaED:2018ArXiv}]\label{pg-sol}
Let (\ref{plsG}) be the equivalent, optimal rank one, representation
of the system (\ref{pls}) and let the matrix $A(\check{p})$ be
nonsingular. Denote $C=A^{-1}(\check{p})$ and
$\check{x}=Ca(\check{p})$. If (\ref{strReg2}) holds true, then
\begin{itemize}
\item[i)] there exists an united {\bf parameterized solution} of the system
(\ref{pls}), (\ref{plsG})
\begin{equation}\label{x(pg)}
x(p) = \check{x} - (CF)\left(\check{p}_{\pi''}-p_{\pi''}\right) +
\left(CLD_{|\inum{y} -
t|}\right)g\left(\check{p}_{\pi'}-p_{\pi'}\right), \quad
 p\in \inum{p},
\end{equation}
where $\inum{y}\supseteq\Sigma_{\rm uni}((\ref{eqY}))$,
\item[ii)] with the same $\inum{y}$ used in (\ref{xx}) and in (\ref{x(pg)}),
interval evaluation $x\left(\inum{p}\right)$ of $x(p)$ is equal to
the interval vector $\inum{x}$ obtained by Theorem \ref{solution1}.
\end{itemize}
\end{theorem}

{\bf 3. Another method for parameterized solution.} In
\cite{Kolev2018} Kolev proposes a parameterized solution based: (a)
on a generalized method of Neumaier and Pownuk \cite{Pop2014CAMWA}
(abbreviated here as iGNP), and (b) on affine arithmetic. It is
reported in \cite{Kolev2018} that the implementation of the proposed
there parameterized method is eight times slower than the interval
method iGNP from \cite{Pop2014CAMWA}. We suppose that the
considerable slow down is due to the affine arithmetic which is used
in both the implementation of iGNP and the parameterized solution
derivation. It is discussed in \cite{PopovaED:2018} that the
proposed there interval method (Theorem \ref{solution1}),
abbreviated as iGRank1, is applicable to the same expanded class of
parametric systems as the method iGNP and provides interval solution
enclosure of the same (sometimes better) quality while overcoming
some specific features that have to be considered in the
implementation of iGNP. In what follows (Theorem
\ref{pl-KolevStyle}) we propose a new parameterized solution,
abbreviated as pKRank1, which is based on Theorem \ref{solution1}
and does not require affine arithmetic. It will be shown (Corollary
\ref{equvKP}) that the interval solution enclosures based on pKRank1
and pPRank1 (Theorem \ref{pg-sol}) are the same in exact arithmetic.
Also, the parameterized solutions pKRank1 and pPRank1 are applicable
to the interval method iGNP with $\inum{y}$ obtained by the latter
(Proposition \ref{NP=sol1}).

\begin{theorem}\label{pl-KolevStyle}
Let (\ref{plsG}) be the equivalent, optimal rank one, representation
of (\ref{pls}) and let $A(\check{p})$ be nonsingular. Denote
$C=A^{-1}(\check{p})$, $\check{x}=Ca(\check{p})$ and let
(\ref{strReg2}) hold true.
\begin{itemize}
\item[(i)] There exists a parameterized solution enclosure of $\Sigma_{uni}((\ref{pls}))$
$$
x(p, r) = \check{x} - (CF)(\check{p}_{\pi''}-p_{\pi''}) +
(CLD_{\check{y}-t})g(\check{p}_{\pi'}-p_{\pi'}) + r, \quad
p\in\inum{p}, r\in \inum{r}=[-\hat{r},\hat{r}],
$$
where $\check{y}=R\check{x}$, $\inum{y}$ is that of Theorem
\ref{solution1} (ii), and
$\hat{r}=|CL|D_{|\inum{y}-\check{y}|}g(\hat{p}_{\pi'})$.
\item[(ii)] The interval evaluation $x(\inum{p}, \inum{r})$ of the function in (i) is equal to the
interval vector $\inum{x}$, obtained by Theorem \ref{solution1},
provided that both vectors are based on the same $\inum{y}$ of
Theorem \ref{solution1} (ii).
\end{itemize}
\end{theorem}
\begin{proof}
Since (\ref{strReg2}) holds true, Theorem \ref{solution1} implies
that every $x\in\Sigma_{uni}\left(A(p), a(p), \inum{p}\right)$
satisfies
\begin{eqnarray} \label{x1(p,y)}
x  \in \inum{x} &=& \check{x} -
(CF)\left(\check{p}_{\pi''}-\inum{p}_{\pi''}\right) +
(CL)\left(D_{g(\check{p}_{\pi'}-\inum{p}_{\pi'})}\left(\inum{y}-t\right)\right) \\ %
 &=& \check{x} -
(CF)\left(\check{p}_{\pi''}-\inum{p}_{\pi''}\right) +
\left(CLD_{|\inum{y}-t|}\right)g(\check{p}_{\pi'}-\inum{p}_{\pi'}).
\label{x2(p,y)}
\end{eqnarray}
Consider the right-hand side in (\ref{x1(p,y)}) as an interval
function $x(p,y)$ of $p\in \inum{p}$, $y\in\inum{y}$ and rearrange
it as follows.
\begin{eqnarray*}
x(p,y)&=&\check{x} - (CF)\left(\check{p}_{\pi''}-p_{\pi''}\right) +
(CL)\left(D_{g(\check{p}_{\pi'}-p_{\pi'})}(\check{y}-t)\right) +
(CL)\left(D_{g(\check{p}_{\pi'}-p_{\pi'})}\left(y-\check{y}\right)\right) \\ %
&=&\check{x} - (CF)\left(\check{p}_{\pi''}-p_{\pi''}\right) +
(CL)\left(D_{\check{y}-t}g(\check{p}_{\pi'}-p_{\pi'})\right) +
(CL)\left(D_{y-\check{y}}g(\check{p}_{\pi'}-p_{\pi'})\right).
\end{eqnarray*}
The interval evaluation $x(\inum{p},\inum{y})$ of the last
expression for $x(p,y)$ is
\begin{eqnarray*}
x(\inum{p},\inum{y}) &=& \check{x} -
(CF)\left(\check{p}_{\pi''}-\inum{p}_{\pi''}\right)
+\left(CLD_{|\check{y}-t|}\right)g(\check{p}_{\pi'}-\inum{p}_{\pi'})+ \left(CLD_{|\inum{y}-\check{y}|}\right)g(\check{p}_{\pi'}-\inum{p}_{\pi'}), %
\end{eqnarray*}
the latter implying the representation (i). In order to prove (ii)
we need to prove that $x(\inum{p},\inum{y})=(\ref{x2(p,y)})$. Since
$|\check{y}-t| - |\inum{y}-\check{y}|\leq |\inum{y}-t| \leq
|\check{y}-t| + |\inum{y}-\check{y}|$,
\begin{multline*}
\left(CLD_{|\check{y}-t|}\right)g(\check{p}_{\pi'}-\inum{p}_{\pi'})-
\left(CLD_{|\inum{y}-\check{y}|}\right)g(\check{p}_{\pi'}-\inum{p}_{\pi'})\leq
\left(CLD_{|\inum{y}-t|}\right)g(\check{p}_{\pi'}-\inum{p}_{\pi'}) \leq \\
\left(CLD_{|\check{y}-t|}\right)g(\check{p}_{\pi'}-\inum{p}_{\pi'})+
\left(CLD_{|\inum{y}-\check{y}|}\right)g(\check{p}_{\pi'}-\inum{p}_{\pi'}).
\end{multline*}
Since $g(\check{p}_{\pi'}-\inum{p}_{\pi'})$ and
$\left(CLD_{|\inum{y}-\check{y}|}\right)g(\check{p}_{\pi'}-\inum{p}_{\pi'})$
are symmetric interval vectors,
$\left(CLD_{|\inum{y}-\check{y}|}\right)g(\check{p}_{\pi'}-\inum{p}_{\pi'})=-
\left(CLD_{|\inum{y}-\check{y}|}\right)g(\check{p}_{\pi'}-\inum{p}_{\pi'})$,
which implies the required assertion and (ii).
\end{proof}

\begin{corollary}\label{equvKP}
Let (\ref{plsG}) be the equivalent, optimal rank one, representation
of (\ref{pls}) and let  $A(\check{p})$ be nonsingular. Denote
$C=A^{-1}(\check{p})$ and $\check{x}=Ca(\check{p})$. If
(\ref{strReg2}) holds true, then
$$x(\inum{p}, \inum{r}) = x(\inum{p}) = \inum{x},$$ where $x(p, r)$ is that of Theorem \ref{pl-KolevStyle},
$x(p)$ is that of Theorem \ref{pg-sol} and $\inum{x}$ is that of
Theorem \ref{solution1}, provided that all computations are in exact
arithmetic and both parameterized solutions use the same $\inum{y}$
of Theorem \ref{solution1} (ii).
\end{corollary}
\begin{proof}The proof is part of the proof of Theorem
\ref{pl-KolevStyle} since $x(\inum{p}) = (\ref{x2(p,y)})=\inum{x}$.
\end{proof}

\begin{proposition}\label{NP=sol1}
Let (\ref{plsG}) be the equivalent, optimal rank one, representation
of (\ref{pls}) and let $A(\check{p})$ be nonsingular. Denote
$C=A^{-1}(\check{p})$ and let (\ref{strReg2}) holds true. If
$\inum{y}$ is obtained by {\rm \cite[Theorem 4]{Pop2014CAMWA}} and
the implementation iteration thereafter, the interval vector
$$\inum{x}_{R1} =
\check{x}-(CF)g(\check{p}_{\pi''}-\inum{p}_{\pi''}) +
(CL)\left(D_{g(\check{p}_{\pi'}-\inum{p}_{\pi'})}(\inum{y}-t)\right),$$
obtained by Theorem \ref{solution1}, and the interval vector
$$\inum{x}_{NP} =
Ca_0 + (CF)\inum{p}_{\pi''} + (CL)\left(D_{g(\check{p}_{\pi'})}t +
D_{g(\check{p}_{\pi'}-\inum{p}_{\pi'})}(\inum{y}-t)\right),$$
obtained by the implementation of {\rm \cite[Theorem
4]{Pop2014CAMWA}}, are equal.
\end{proposition}
\begin{proof}
In the notation of \cite{Pop2014CAMWA}, $D_0={\rm
Diag(g(\check{p}_{\pi'}))}= D_{g(\check{p}_{\pi'})}$,
$\left[D\right]= D_{g(\inum{p}_{\pi'})}$,
$[d]:=\left(D_0-\left[D\right]\right)(\inum{y}-t)=
D_{g(\check{p}_{\pi'}-\inum{p}_{\pi'})}(\inum{y}-t)$. Then, due to
$\check{x}=Ca(\check{p}) = Ca_0 +CF\check{p}_{\pi''} +
CLD_{g(\check{p}_{\pi'})}t$, we have the desired equality.
\end{proof}

Proposition \ref{NP=sol1} implies that the two kinds of
parameterized solutions, obtained by Theorem \ref{pg-sol} and
Theorem \ref{pl-KolevStyle}, are applicable to the generalized
method of Neumaier and Pownuk \cite{Pop2014CAMWA} (iGNP) with
$\inum{y}$ obtained by the latter method. Since \cite{PopovaED:2018}
reports for better solution enclosures provided by Theorem
\ref{solution1} compared to iGNP for some problems, as well as for a
better performance in a computing environment, it is expected that
these advantages will be attributable to the above two kinds of
parameterized solutions, obtained by Theorem \ref{pg-sol} and
Theorem \ref{pl-KolevStyle}. One advantage of the parameterized
solutions involving the remainder term $r\in [-\hat{r},\hat{r}]$  is
that they allow obtaining an inner estimate of the hull solution,
presented in the next proposition.

\begin{proposition}\label{inner}
Let $x(q)=\check{x}+Uq+r$, $q\in [-\hat{q},\hat{q}]=\inum{q}$, $r\in
[-\hat{r},\hat{r}]=\inum{r}$, be the parameterized solution obtained
by Theorem \ref{pl-KolevStyle}, where
$U=\left(-CF,CLD_{\check{y}-t}\right)$, $q=(p_{\pi''}^\top,
g^\top(p_{\pi'}))^\top$, $\hat{q}=(\hat{p}_{\pi''}^\top,
g^\top(\hat{p}_{\pi'}))^\top$. Define
\begin{eqnarray*}
\inum{v}^{\rm Low} := \check{x}+\left( U\inum{q}\right)^- +
\inum{r}, &\quad & \inum{v}^{\rm Up} := \check{x}+\left(
U\inum{q}\right)^+ + \inum{r}.
\end{eqnarray*}
With $\inum{x}_* = \left[x_*^-, x_*^+\right] = \square \Sigma^p_{\rm
uni}((\ref{pls}))$,
\begin{eqnarray*}
x_*^- \in \inum{v}^{\rm Low}, \; x_*^+ \in \inum{v}^{\rm Up},
&\text{that is } & \inum{x}_*\subseteq \inum{v}^{\rm Low} \cup
\inum{v}^{\rm Up} = x(\inum{q}).
\end{eqnarray*}
Define $\inum{x}_{\rm in}(\inum{q}) := \check{x}+U\inum{q} +
\inum{r}_-$, where the interval evaluation is in Kaucher interval
arithmetic {\rm \cite{Kaucher80}} and $\inum{r}_-$ denotes
$\dual{\inum{r}}$. In classical interval arithmetic
$$\left(\inum{x}_{\rm in}(\inum{q})\right)^- = \check{x}+U\inum{q} + \hat{r}, \quad
\left(\inum{x}_{\rm in}(\inum{q})\right)^+ = \check{x}+U\inum{q} -
\hat{r}.
$$
For every $i$, $1\leq i\leq n$, such that $\left(\inum{x}_{\rm
in}(\inum{q})\right)^-_i > \left(\inum{x}_{\rm
in}(\inum{q})\right)^+_i$, substitute $\left(\inum{x}_{\rm
in}(\inum{q})\right)_i = \emptyset$. Then, it holds $\inum{x}_{\rm
in}(\inum{q}) \subseteq   \square \Sigma^p_{\rm uni}((\ref{pls}))
\subseteq x(\inum{q})$.
\end{proposition}
\begin{proof}
The proof can be based on the properties of Kaucher interval
arithmetic \cite{Kaucher80}, or to be done similarly to that of
\cite[Theorem 1]{Kol2014}.
\end{proof}

The methodology in \cite{PopovaED:2018} has an expanded scope of
applications for systems involving rank one uncertainty structure.
Next example demonstrates the advantage of the proposed here
parameterized solution (Theorem \ref{pl-KolevStyle}) to a variety of
parameterized solutions based on affine arithmetic and compared in
\cite{SkalnaHla19}.

\begin{example}\label{Okum}\rm
Consider the parametric linear system
\begin{equation*}
\begin{pmatrix}p_1 + p_6 & -p_6 & 0 & 0 & 0\\
-p_6 & p_2 + p_6 + p_7 & -p_7 & 0 & 0 \\ %
0 & -p_7 & p_3 + p_7 + p_8 & -p_8 & 0 \\ %
0 & 0 & -p_8 & p_4 + p_8 + p_9 & -p_9 \\ %
0 & 0 & 0 & -p_9 & p_5 + p_9
 \end{pmatrix}x = \begin{pmatrix}10\\ 0\\ 10\\ 0\\ 0 \end{pmatrix}
\end{equation*}
after \cite{Okumura}, where the parameters vary within given
intervals $p_i\in [1-\delta, 1+\delta]$. This example is considered
in \cite[Example 5]{SkalnaHla19} and the outer solution enclosures
obtained by six parameterized solutions based on affine arithmetic
are compared to the direct parameterized method (abbreviated PDM) of
\cite{Kol2016a}. Here we compare the parameterized inner and outer
bounds for the solution set, obtained by the method of Theorem
\ref{pl-KolevStyle} and Proposition \ref{inner}, and the
corresponding bounds obtained by \cite{Kol2016a}, thus comparing to
the other six parameterized solutions considered in
\cite{SkalnaHla19}. We present the results for the smallest
uncertainty $\delta=0.01$ and the largest uncertainty $\delta=0.25$,
considered in \cite{SkalnaHla19}.
\end{example}

Table \ref{Table1} presets inner and outer bounds obtained by us for
$\delta=0.01$. These bounds are much sharper than, and can be
compared to, the bounds obtained by three other parameterized
solutions reported in \cite[Table 4]{SkalnaHla19}. For the results
in Table \ref{Table1}, Table \ref{Table1a} presents two measures of
the quality of a solution enclosure: sharpness $O_s$ of the solution
enclosure $\inum{x}_{\rm out}$ defined by $Q_s(\inum{x}_{\rm in},
\inum{x}_{\rm out}) :=\left\{0 \text{ if } \inum{x}_{\rm
in}=\emptyset, {\rm rad}(\inum{x}_{\rm in})/{\rm rad}(\inum{x}_{\rm
out}) \text{ otherwise}\right\}$, and percentage $O_w$ by which an
interval $\inum{y}$ overestimates the interval $\inum{x}$,
$\inum{x}\subseteq\inum{y}$, defined by $O_w(\inum{x}, \inum{y}) :=
\left(1 - {\rm rad}(\inum{x})/{\rm rad}(\inum{y})\right) 100$. It is
seen from Table \ref{Table1a} that the range of the sharpness
measure is very close for the two methods pKRank1 and PDM. On the
other hand, the percentage by which ${\rm PDM}_{\rm out}$
overestimates ${\rm pKRank1}_{\rm out}$ is between $0.55$\% and
$0.96$\%. Table \ref{Table25} presents the two measures of the
quality of a solution enclosure for the case of large parameter
uncertainties $\delta = 0.25$ in Example \ref{Okum}. Although the
percentage by which ${\rm PDM}_{\rm out}$ overestimates ${\rm
pKRank1}_{\rm out}$ is more pronounced in this case, the ranges of
sharpness is very close for these two methods and the methods
compared in \cite{SkalnaHla19}. The first conclusion from Example
\ref{Okum} is that the methods based on condition (\ref{strReg2})
provide sharper solution enclosure than the methods based on
condition (\ref{rhoDelta}) for systems with rank one uncertainty
structure. The second important conclusion from this example is that
the sharpness measure is not quite informative when comparing the
solution enclosure of different methods in contrast to the
percentage of overestimation.

\begin{table}[ht]
{\small
\begin{tabular}{ccccc}\hline
$\inum{x}$ & \multicolumn{2}{c}{outer} & \multicolumn{2}{c}{inner}\\
\hline
         & {\small pKRank1} & PDM & {\small pKRank1} & PDM  \\
         \hline
$x_1$ & $[7.01522, 7.16659]$ & $[7.01480, 7.16702]$ & $[7.01736, 7.16446]$ & $[7.01777, 7.16405]$ \\%
$x_2$ & $[4.11780, 4.24583]$ & $[4.11736, 4.24628]$ & $[4.11987, 4.24377]$ & $[4.12030, 4.24333]$ \\%
$x_3$ & $[5.39374, 5.51535]$ & $[5.39331, 5.51578]$ & $[5.39567, 5.51342]$ & $[5.39609, 5.51300]$ \\%
$x_4$ & $[2.13805, 2.22558]$ & $[2.13770, 2.22594]$ & $[2.13962, 2.22401]$ & $[2.13997, 2.22367]$ \\%
$x_5$ & $[1.06046, 1.12136]$ & $[1.06017, 1.12165]$ & $[1.06171, 1.12011]$ & $[1.06200, 1.11982]$ \\%
\hline
\end{tabular}
} \caption{Bounds for $\square\Sigma$ in Example \ref{Okum},
$\delta=0.01$, obtained by  pKRank1 and PDM.} \label{Table1}
\end{table}

\begin{table}[h]
\centerline{
\begin{tabular}{l|cccccc}\hline
& $x_1$ & $x_2$ & $x_3$ & $x_4$ & $x_5$ & range in \cite{SkalnaHla19}\\ \hline %
$O_s$, {\small pKRank1} & 0.972 & 0.968 & 0.968 & 0.964 & 0.959 \\ %
$O_s$, {\small PDM}     & 0.961 & 0.954 & 0.954 & 0.948 & 0.940 & 0.95--0.97 \\   \hline %
\% overest.      & 0.555 & 0.692 & 0.702 & 0.799 & 0.959 \\ \hline %
\end{tabular}
} \caption{Sharpness $O_s$ for pKRank1 and PDM for the bounds in
Table \ref{Table1} and the percentage by which ${\rm PDM}_{\rm out}$
overestimates ${\rm pKRank1}_{\rm out}$.} \label{Table1a}
\end{table}

\begin{table}[h]
\centerline{
\begin{tabular}{l|cccccc}\hline
& $x_1$ & $x_2$ & $x_3$ & $x_4$ & $x_5$ & range in \cite{SkalnaHla19}\\ \hline %
$O_s$, {\small pKRank1} & 0.266 & 0.189 & 0.186 & 0.113 & 0.028 \\ %
$O_s$, {\small PDM}     & 0.05 & 0 & 0 & 0 & 0 & 0.0--0.26 \\   \hline %
\% overest.             & 24.3 & 27.7 & 28.9 & 31.9 & 35.2 \\ \hline %
\end{tabular}
} \caption{Sharpness $O_s$ of pKRank1 and PDM for the bounds of the
solution set in Example \ref{Okum}, $\delta = 0.25$,   and the
percentage by which ${\rm PDM}_{\rm out}$ overestimates ${\rm
pKRank1}_{\rm out}$.} \label{Table25}
\end{table}

{\bf 4. On an application of pKRank1.} In \cite[Section
3]{Kolev2018} Kolev proposes to apply the parameterized solution of
type Theorem \ref{pl-KolevStyle} for determining $\square\Sigma^p$
of parametric systems involving rank one interval parameters. In
this section we consider such an application in more details and
demonstrate that this might be dangerous.

Let (\ref{plsG}) be the equivalent, optimal rank one, representation
of (\ref{pls}), which involves only rank one interval parameters.
Let  $A(\check{p})$ be nonsingular. Denote $C=A^{-1}(\check{p})$,
$\check{x}=Ca(\check{p})$, $\check{y}=R\check{x}$, and let
(\ref{strReg2}) hold true. Let $i$ be arbitrary, $1\leq i\leq n$,
and let $\left(\Sigma_{uni}((\ref{pls}))\right)_i$ be monotone with
respect to each parameter $p_k$, $k\in\pi = ((\pi'')^\top ,
(\pi')^\top)^\top$, so that
$$\left(\square\Sigma_{uni}((\ref{pls}))\right)_i =
[x_{*,i}^-, x_{*,i}^+] = \left[x_i(p^{-s_i}), x_i(p^{s_i})\right]$$
for an $s_i\in\{-1,1\}^K$ ($|s_i|=1\in\R^K$), where $-1$ means
decreasing and $1$ -- increasing. In order to simplify the notation,
in what follows we will omit the subscript in $s_i$. Denote
$s_i^\top = s^\top =(s'',s')$. We consider $\inum{x}$ in Theorem
\ref{solution1} as an interval evaluation of the function
\begin{eqnarray*}
x(p) &=& \check{x}-(CF)(\check{p}_{\pi''}-p_{\pi''})
+(CL)\left(D_{\check{p}_{\pi'}-p_{\pi'}}(y(p)-t)\right),
\end{eqnarray*}
where $y(p)$ is the solution of the system (\ref{eqY}). Replacing in
this function the two endpoint vectors $p^{-s}$, respectively
$p^{s}$, we obtain
\begin{eqnarray*}
x_i(p^{-s}) &=&
\check{x}_i-C_{i\bullet}F(\check{p}_{\pi''}-p^{-s''}_{\pi''})
+C_{i\bullet}LD_{\check{y}-t}\,(\check{p}_{\pi'}-p_{\pi'}^{-s'}) + C_{i\bullet}LD_{y(\check{p} -p^{-s})-\check{y}}(\check{p}_{\pi'}-p_{\pi'}^{-s'})\\ %
x_i(p^{s}) &=&
\check{x}_i-C_{i\bullet}F(\check{p}_{\pi''}-p^{s''}_{\pi''})
+ C_{i\bullet}LD_{(\check{y}-t)}(\check{p}_{\pi'}-p_{\pi'}^{s'}) + C_{i\bullet}LD_{y(\check{p}-p^{s})-\check{y}}\,(\check{p}_{\pi'}-p_{\pi'}^{s'}). %
\end{eqnarray*}
In order to simplify the presentation, we denote $\lambda'' =
C_{i\bullet}F$, $\lambda' = C_{i\bullet}LD_{\check{y}-t}$. %
For $\inum{p}\in\IR$ and $s\in\{-1,1\}$, we have $\inum{p}^{-s} =
\check{p} - s\hat{p}$, $\inum{p}^{s} = \check{p} + s\hat{p}$ and
\begin{eqnarray*}
\check{p} - \inum{p}^{-s} = \check{p} -(\check{p} - s\hat{p}) =
s\hat{p}, \quad &\text{similarly,}& \quad \check{p} - \inum{p}^{s} =
-s\hat{p}.
\end{eqnarray*}
Thus, we have
\begin{eqnarray*}
x_i(p^{-s}) &=& \check{x}_i-\lambda''(s''\hat{p}_{\pi''})
+\lambda'(s'\hat{p}_{\pi'}) +
C_{i\bullet}LD_{y(s\hat{p}_\pi)-\check{y}}\,(s'\hat{p}_{\pi'}),\\ %
x_i(p^{s}) &=& \check{x}_i-\lambda''(-s''\hat{p}_{\pi''})
+\lambda'(-s'\hat{p}_{\pi'}) +
C_{i\bullet}LD_{y(-s\hat{p}_\pi)-\check{y}}\,(-s'\hat{p}_{\pi'}).
\end{eqnarray*}
Now, in order to operate simultaneously with both $p^{-s}$, $p^{s}$,
as well as simultaneously with both $x_i(p^{-s})$, $x_i(p^{s})$, we
use Kaucher interval arithmetic and the relations between proper and
improper intervals. Consider the following interval expression in
$\KR$
\begin{equation} \label{x(ps)}
\check{x}_i -\lambda''(\inum{p}'_{\pi''})_{-s''} + \lambda'(\inum{p}'_{\pi'})_{-s'} %
+  [r_{*,i}^-, r_{*,i}^+],
\end{equation}
where $[r_{*,i}^-, r_{*,i}^+] =
\left[C_{i\bullet}LD_{y(s\hat{p}_\pi)-\check{y}}\,(s'\hat{p}_{\pi'}),
\;
C_{i\bullet}LD_{y(-s\hat{p}_\pi)-\check{y}}\,(-s'\hat{p}_{\pi'})\right]$
and $\inum{p}'_\pi = [-\hat{p}_\pi, \hat{p}_\pi]$. (\ref{x(ps)}) is
equivalent to
\begin{equation} \label{x(ps)2}
\check{x}_i -|\lambda''|(\inum{p}'_{\pi''})_{-s''s_{\lambda''}} + |\lambda'|(\inum{p}'_{\pi'})_{-s's_{\lambda'}} %
+  [r_{*,i}^-, r_{*,i}^+].
\end{equation}
If
\begin{eqnarray} \label{cond}
s_{\lambda''}:={\rm sign}(\lambda'') = -s'' &\text{and}&
s_{\lambda'}:={\rm sign}(\lambda') = -s',
\end{eqnarray}
then
\begin{eqnarray*}
-\lambda''(\inum{p}'_{\pi''})_{-s''} &=&
-s_{\lambda''}|\lambda''|(\inum{p}'_{\pi''})_{-s''} =
-|\lambda''|(\inum{p}'_{\pi''})_{-s''s_{\lambda''}}
= -|\lambda''|(\inum{p}'_{\pi''}) \\
&=& -[-|\lambda''|\hat{p}_{\pi''},
|\lambda''|\hat{p}_{\pi''}]\stackrel{(\ref{cond})}{=}
-[s''\lambda''\hat{p}_{\pi''}, -s''\lambda''\hat{p}_{\pi''}],
\end{eqnarray*}
similarly $\lambda'(\inum{p}'_{\pi'})_{-s'} =
[\lambda's'\hat{p}_{\pi'}, -\lambda's'\hat{p}_{\pi'}]$. Thus,
(\ref{x(ps)}), (\ref{x(ps)2}), become equivalently
\begin{multline}\label{equal}
\check{x}_i -[\lambda''(s''\hat{p}_{\pi''}),
\lambda''(-s''\hat{p}_{\pi''})] +
 [\lambda'(s'\hat{p}_{\pi'}), -\lambda'(s'\hat{p}_{\pi'})] %
+  [r_{*,i}^-, r_{*,i}^+] = \\
 \left[x_i(p^{-s_i}), x_i(p^{s_i})\right].
\end{multline}
Thus, by (\ref{cond}), (\ref{x(ps)2}) is equivalent to \quad
$\check{x}_i -\lambda''\inum{p}'_{\pi''} +
\lambda'\inum{p}'_{\pi'} + [r_{*,i}^-, r_{*,i}^+]$. \\
Now, we compare (\ref{x(ps)2}) to $x_i(\inum{p},\inum{r})$, where
$x(p,r)$ is the parameterized solution from Theorem
\ref{pl-KolevStyle}. The coefficients $\lambda''$, $\lambda'$ are
the same in both expressions. Consider three cases.
\begin{itemize}
\item
Obviously, under (\ref{cond}), the first three terms in the two
expressions are equivalent.

\item
If for some $k\in\pi$, $\lambda_k =0$, then the equality relation
(\ref{equal}) is preserved and the equivalence between the first
three terms in  (\ref{x(ps)2}) and $x_i(\inum{p},\inum{r})$ is also
preserved. However,  $s_k\neq s_{\lambda_k}$ and $s_k$ cannot be
inferred from $\lambda_k$.

\item
If for some $k\in\pi$, $0\neq s_{\lambda_k} \neq -s_k$, the equality
relation (\ref{equal}) turns into inclusion (due to
$|\lambda_k|\left(\inum{p}'_k\right)_- \subseteq
|\lambda_k|\left(\inum{p}'_k\right)$), which contradicts to the
initial assumption. In this case, the first three terms in
(\ref{x(ps)2}) and $x_i(\inum{p},\inum{r})$ are equivalent but $s_k$
also cannot be inferred from $\lambda_k$.
\end{itemize}

Thus, we have proven the following theorem.
\begin{theorem}\label{thmMon1}
Let (\ref{plsG}) be the equivalent, optimal rank one, representation
of the system (\ref{pls}), which involves only rank one interval
parameters. Let the matrix $A(\check{p})$ be nonsingular. Denote
$C=A^{-1}(\check{p})$, $\check{x}=Ca(\check{p})$,
$\check{y}=R\check{x}$, and let the condition (\ref{strReg2}) hold
true. If for any $i$, $1\leq i\leq n$,
$\left(\Sigma_{uni}((\ref{pls}))\right)_i$ is monotone with respect
to each parameter $p_k$, $k\in\pi = ((\pi'')^\top ,
(\pi')^\top)^\top$, with type of monotonicity specified by the sign
vector $s_i$, and  if ${\rm sign}\left(CF, \;
CLD_{\check{y}-t}\right)_{i\bullet} = s_i$, then the parameterized
solution defined in Theorem \ref{pl-KolevStyle} can be used for
determining $\square\left(\Sigma_{\rm uni}((\ref{pls}))\right)_i$.
\end{theorem}

It follows from (\ref{equal})  that with given $s_i$,
$\square\left(\Sigma_{\rm uni}((\ref{pls}))\right)_i$ can be
obtained by solving (\ref{pls}) for $\inum{p}^{-s_i}$, respectively,
for $\inum{p}^{s_i}$, $\inum{p}\in\IR^K$, or by solving the
equivalent centered system
$$\left(A(\check{p}) -LD_{p'_{\pi'}}R\right)x = a(\check{p}) -Fp'_{\pi''} - LD_{p'_{\pi'}}t, \qquad p'\in [-\hat{p},\hat{p}]$$
for $p'=s_i\hat{p}$, respectively, for $p'=-s_i\hat{p}$, that is
$(\inum{p}')^{s_i}$, resp., $(\inum{p}')^{-s_i}$.

Note, that the matrix $\left(CF, \; CLD_{\check{y}-t}\right)$ is
different from the matrix $\left(-CF,\right.$
$\left.CLD_{\check{y}-t}\right)$ in the parameterized solution of
Theorem \ref{pl-KolevStyle}. Note also, that the interval
$[r_{*,i}^-, r_{*,i}^+]\in\IR$ is not symmetric in general and
differs from the symmetric interval $[-r,r]$ in Theorem
\ref{pl-KolevStyle}.

By Proposition \ref{NP=sol1}, it follows that (for rank one
uncertainty structure of the system) the parameterized solution
obtained by the method of \cite{Kolev2018} (based on affine
arithmetic) will have the same signs of the parameter coefficients
as the parameterized solution of Theorem \ref{pl-KolevStyle}. The
example, considered in \cite{Kolev2018}, illustrates the first case
($s_{\lambda} = s_i$) in the proof of Theorem \ref{thmMon1}. By the
following example we illustrate the last case ($0\neq s_{\lambda_k}
\neq -s_k$) in the proof of Theorem \ref{thmMon1}, which implies
that the parameterized solution of type Theorem \ref{pl-KolevStyle}
cannot be used in general for determining the hull solution to
interval parametric linear systems involving rank one parameters.

\begin{example}\rm
Consider the parametric linear system
$$
\begin{pmatrix}1&\frac{1}{4}+p_{12}&p_{22}\\ \frac{1}{4}+p_1&2&\frac{1}{4}+p_1\\ p_2&\frac{1}{4}+p_{12}&3\end{pmatrix}x =
\begin{pmatrix}-\frac{5}{2} - p_3\\ \frac{8}{3}+\frac{p_3}{3}\\ -\frac{9}{4} +
\frac{p_3}{2}\end{pmatrix}, \qquad \begin{matrix}p_1,p_{12}\in
[-\frac{3}{4},\frac{3}{4}],\\p_2,p_{22}\in [-\frac{1}{2},\frac{1}{2}],\\
p_3\in [-\frac{1}{2},\frac{1}{2}].\end{matrix}
$$
The coefficient matrix for everyone of the parameters has rank one.
Therefore, by \cite[Corollary 3]{Pop2014CAMWA} the parametric united
solution set has linear boundary and its interval hull is obtained
for particular endpoints of the parameter intervals. Table
\ref{exMon} (right) presents global monotonicity (single entry) or
local monotonicity (two entries respectively for the lower and the
upper bounds) type of the parametric solution set with respect to
interval parameters.
An equivalent representation of the system is defined by
\begin{equation*}
L=\begin{pmatrix}0&1&0&1\\1&0&0&0\\0&1&1&0\end{pmatrix},\quad
R=\begin{pmatrix}1&0&1\\0&1&0\\1&0&0\\0&0&1\end{pmatrix}, \quad
\begin{matrix}g(p)= (p_1,p_{12},p_2, p_{22})^\top,\\ F=(-1, \frac{1}{3},  \frac{1}{2})^\top, \\
t=0.\end{matrix}
\end{equation*}
Applying Theorem \ref{pl-KolevStyle} we obtain a parameterized
solution with reminder term  $x(p,r) = \check{x}+ U (p_3,
p_1,p_{12},p_2, p_{22})^\top + r$,  where $r\in [-\hat{r},\hat{r}]$
and
\begin{eqnarray*}
U&=&\begin{pmatrix}
 1.07065 & 0.502836 & 1.89414 & -0.0321066 & -0.930657\\
 -0.282609 & -2.01134 & -0.31569 & 0.128426, & 0.117557\\
-0.143116 & 0.167612 & 0.63138 & -0.995304 & -0.00979639
\end{pmatrix}, \\
\check{x}&=&\left(-2.9538, 1.81522, -0.901268\right)^\top, \quad
\hat{r}=\left(52.7807, 39.2595, 22.8547\right)^\top.
\end{eqnarray*}

\begin{table}[h]
\begin{tabular}{lccccc}
\hline  & $p_3$ & $p_1$ & $p_{12}$ & $p_2$ & $p_{22}$ \\
\hline
$x_1$& -1 & 1 & 1 & -1 & -1 \\
$x_2$& 1 & -1 & -1 & 1 & 1 \\
$x_3$& 1 & 1 & 1 & -1 & -1
\end{tabular} \qquad\qquad
\begin{tabular}{lccccc}
\hline
     & $p_3$ & $p_1$ & $p_{12}$ & $p_2$ & $p_{22}$ \\ \hline
$x_1$& 1, -1 & 1     & 1, -1    & -1    & -1, 1 \\
$x_2$&  1    & -1, 1  & 1        & -1    & -1 \\
$x_3$&  1    & 1     & 1, -1    & -1, 1 & -1, 1
\end{tabular}
\caption{Monotonic dependence of the solution components on the
interval parameters. Left: sign$(U)$ in the parameterized solution;
Right: true dependence.}\label{exMon}
\end{table}
The sign of matrix $U'=\left(CF, \; CLD_{\check{y}-t}\right)$ is
presented in Table \ref{exMon} (left). Comparing the left and right
tables of monotonicity, it is clear that the sign of matrix $U'$
does not represent the true monotonic dependence. Furthermore, no
one of the elements of $U, U'$ is zero. Applying the monotonicity
defined by sign$(U')$ we obtain an interval vector, which is
contained in $\square\Sigma^p$. Therefore, using sign$(U')$ is
dangerous.
\end{example}

{\bf Acknowledgements} This work is partly supported by the National
Scientific Program ``Information and Communication Technologies for
a Single Digital Market in Science, Education and Security
(ICTinSES)'', contract No DO1-205/23.11.2018, financed by the
Ministry of Education and Science in Bulgaria.

\bigskip

\normalsize

\noindent Institute of Mathematics and Informatics,
Bulgarian Academy of Sciences  \\
Acad. G. Bonchev str., block 8, 1113 Sofia, Bulgaria       \\
email: {\tt epopova@math.bas.bg}

 \end{document}